\documentclass[11 pt, leqno]{amsart}

\usepackage{amsmath, amsthm, amssymb, amsfonts}
\usepackage{a4wide}

\numberwithin{equation}{section}

\newcommand{\id}{\mathord{\operatorname{id}}}
\newcommand{\Z}{\mathbf{Z}}
\newcommand{\N}{\mathbf{N}}
\newcommand{\rE}{\operatorname{E}}
\newcommand{\rV}{\operatorname{V}}

\newcommand{\Gr}{\mathcal{G}}
\newcommand{\EG}{\rE(\Gr)}
\newcommand{\VG}{\rV(\Gr)}
\theoremstyle{plain}
\newtheorem{theorem}{Theorem}[section]
\newtheorem*{theoA}{Theorem A}

\newtheorem{lemma}[theorem]{Lemma}

\theoremstyle{definition}
\newtheorem{definition}[theorem]{Definition}
\newtheorem{example}[theorem]{Example}

\newtheorem{remark}[theorem]{Remark}

\begin{document}
%%%%%%%%%%%%%%%%%%%%%%%%%%%%%%%%%%%%%%%%%%%%%%%%%%%%%%%%%%%%%%%%%%%%%%%%%%%%%%%%%%%%%%%%%%
%\maketitle

\author{Pierre Fima}
\address{Pierre Fima
\newline
Universit\'e Denis-Diderot - Paris 7, Institut Math\'ematiques de Jussieu, CNRS UMR 7586, Site Sophie Germain, 75013 Paris, France.}
\email{pfima@math.jussieu.fr}
 \thanks{The first author is partially supported by ANR Grants OSQPI and NEUMANN}
\author{Soyoung Moon}
\thanks{The second author is partially supported by the Conseil Regional de Bourgogne (Faber 2012-1-9201-247)}
\address{Soyoung Moon
\newline
Universit\'e de Bourgogne, Institut Math\'ematiques de Bourgogne, CNRS UMR 5584, B.P. 47870, 21078 Dijon Cedex, France.}
\email{soyoung.moon@u-bourgogne.fr}
\author{Yves Stalder}
\address{Yves Stalder
\newline
Clermont Universit\'e, Universit\'e Blaise Pascal, Laboratoire de Math\'ematiques, BP 10448,\linebreak[4] F-63000 Clermont-Ferrand, France.
\newline
CNRS UMR 6620, LM, F-63171 Aubi\`ere, France.}
\email{yves.stalder@math.univ-bpclermont.fr}

\title{Highly transitive actions of groups acting on trees}
\begin{abstract}

\noindent We show that a group acting on a non-trivial tree with finite edge stabilizers and icc vertex stabilizers admits a faithful and highly transitive action on an infinite countable set. This result is actually true for infinite vertex stabilizers and some more general, finite of infinite,  edge stabilizers  that we call highly core-free. We study the notion of highly core-free subgroups and give some examples. In the case of amalgamated free products over highly core-free subgroups and HNN extensions with highly core-free base groups we obtain a genericity result for faithful and highly transitive actions. In particular, we recover the result of D. Kitroser stating that the fundamental group of a closed, orientable surface of genus $g>1$ admits a faithful and highly transitive action.
\end{abstract}

\maketitle

\noindent An action of a countable group $\Gamma$ on an infinite countable set $X$ is called \textit{highly transitive} if it is $n$-transitive for all $n\geq 1$. It is easy to see that an action $\Gamma\curvearrowright X$ is highly transitive if and only if the image of $\Gamma$ in the Polish group $S(X)$ of bijections of $X$ is dense.

\vspace{0.2cm}

\noindent An obvious example of a highly transitive and faithful action is given by the action of the countable group of finitely supported permutations of $X$. This group is not finitely generated and far from being free (it is amenable).

\vspace{0.2cm}

\noindent The first explicit construction of a highly transitive and faithful action of the free group $\mathbf{F}_n$, for $2\leq n\leq\infty$, was published in \cite{McD77} by T.P. McDonough. Then, J. D. Dixon \cite{Di90} showed that most (in a topological sense) finitely generated highly transitive subgroups of $S(X)$ are free.

\vspace{0.2cm}

\noindent A.M.W. Glass and S.H. McCleary \cite{GM91} have constructed faithful and highly transitive action of a free product $\Gamma_1*\Gamma_2$ of non-trivial countable (or finite) groups with $\Gamma_2$ having an element of infinite order. They also observed that $\Z_2*\Z_2$ does not have a faithful and $2$-transitive action and they asked for which non-trivial groups $\Gamma_i$, $i=1,2$, does $\Gamma_1*\Gamma_2$ have a faithful and highly transitive action.

\vspace{0.2cm}

\noindent S.V. Gunhouse \cite{Gu92} completely answered the question by constructing, for any non-trivial countable (or finite) groups $\Gamma_i$, $i=1,2$, with one of the $\Gamma_i$ of size at least $3$, a faithful and highly transitive action of $\Gamma_1*\Gamma_2$. A similar result was obtained independently by K.K. Hickin \cite{Hi92}. Recently, the last two authors \cite{MoSt} proved a genericity result for faithful and highly transitive actions of free products.

\vspace{0.2cm}

\noindent Using the techniques of Hickin, Gunhouse also characterized in his PhD (unpublished result) the non-trivial amalgamated free products with amalgamation over an Artinian group\footnote{that is, any decreasing chain of its distinct subgroups terminates after a finite number; e.g. finite groups.}, that admit a faithful and highly transitive action. He proved that if $\Gamma=\Gamma_1\underset{\Sigma}{*}\Gamma_2$ and $\Sigma$ is an Artinian group, properly included in $\Gamma_1$ and $\Gamma_2$, then $\Gamma$ admits a faithful and highly transitive action if and only if the only subgroup of $\Sigma$ which is normal in both $\Gamma_1$, $\Gamma_2$ is the trivial group. By the results of \cite{Co09}, this last condition is equivalent, when $\Sigma$ is finite and has index at least three in one of the $\Gamma_i$, to say that $\Gamma$ is icc.

\vspace{0.2cm}

\noindent Other examples of groups admitting a faithful and highly transitive action were given recently: surface groups (D. Kitroser \cite{Ki12}), ${\rm Out}(\mathbf{F}_n)$, for $n\geq 4$ (S. Garion and Y. Glasner \cite{GG11}), and non-elementary hyperbolic groups with trivial finite radical (V. V. Chainikov, \cite[Section IV.4]{ChaPhD}).

\vspace{0.2cm}

\noindent In this paper we are interested in groups acting without inversion on a tree (e.g. amalgamated free product and HNN-extensions) and admitting a faithful and highly transitive action.

\vspace{0.2cm}

\noindent Observe that, by \cite[Proposition 1.4]{MoSt}, if $\Gamma$ admits a highly transitive and faithful action then $\Gamma$ is icc.

\vspace{0.2cm}

\noindent We call a tree \textit{non-trivial} if it has at least two edges, $e$ and its inverse edge $\overline{e}$. For general groups acting on trees we obtain the following result.

\begin{theoA}
Let $\Gamma$ be a group acting without inversion on a non-trivial tree. If the vertex stabilizers are icc and the edge stabilizers are finite then $\Gamma$ admits a faithful and highly transitive action.
\end{theoA}

\vspace{0.2cm}

\noindent We obtain actually a more general statement (Theorem \ref{Main}) including some infinite edge stabilizers that we call highly core-free subgroups (Definition \ref{defhcf}).

\vspace{0.2cm}

\noindent To prove Theorem A we use standard arguments from Bass-Serre theory to reduce the case to either an amalgamated free product or an HNN-extension. For this two cases we prove a genericity result for faithful and highly transitive actions  (Theorem \ref{ThmHNN} and Theorem \ref{ThmFree}). We recover, as a particular case of our result on amalgamated free product, the result of D. Kitroser about surface groups (Example \ref{Kitroser}).

\vspace{0.2cm}

\noindent The paper is organized as follows. Section 1 is a preliminary section in which we introduce and study the notion of highly core-free subgroups that we use in the paper. In section $2$ we study the case of HNN-extensions and the case of amalgamated free products is treated in section $3$. We prove Theorem A in section $4$. Finally, we give some examples and links with former results in section $5$.

\section{Highly core-free subgroups}

\vspace{0.2cm}

\noindent Let $\Sigma$ be a subgroup of a group $H$. For a subset $S\subset H$, the \textit{normal core of $\Sigma$ with respect to $S$} is defined by Core$_{S}(\Sigma)=\cap_{h\in S}h^{-1}\Sigma h$. Recall that $\Sigma$ is called \textit{a core-free subgroup} if Core$_H(\Sigma)=\{1\}$. We need a stronger condition than core-freeness namely, we ask that, for every covering of $H$ with non-empty sets, there exists at least one set in the covering for which the associated normal core of $\Sigma$ is trivial. Asking for this property to hold also for coverings up to finitely many $\Sigma$-classes leads to the following definition.

\begin{definition}\label{defhcf} A subgroup $\Sigma< H$ is called \textit{highly core-free} if, for every finite subset $F\subset H$, for any $n\geq 1$, for any non-empty subsets $S_1,\ldots, S_n\subset H$ such that $H\setminus\Sigma F\subset\cup_{k=1}^n S_k$ there exists $1\leq k\leq n$ such that Core$_{S_k}(\Sigma)=\{1\}$.
\end{definition}

\vspace{0.2cm}

\noindent It is clear that a highly core-free subgroup is core-free. Also, if $H$ is finite then every non-trivial subgroup is not highly core-free (however there exists finite groups with non-trivial core-free subgroups, for example the permutation group $S(n)$ in $S(n+1)$ is core-free). More generally, a finite index subgroup is never highly core-free. Indeed, if $\Sigma$ is a non-trivial finite index subgroup of $H$ and $F$ is a finite subset such that $H=\Sigma F$, take $n=1$ and $S_1=\{1\}$ so that we have Core$_{S_1}(\Sigma)=\Sigma\neq\{1\}$. This shows that the highly core-free condition is interesting only for subgroups of infinite groups. Examples of highly core-free subgroups will be given in Example \ref{Exhcf}.

\vspace{0.2cm}

\noindent Recall that a subgroup $\Sigma< H$ is core-free if and only if the action $H\curvearrowright H/\Sigma$ on the left cosets is faithful (this argument also shows that a core-free subgroup of an infinite group has infinite index). The highly core-free condition will be equivalent to a stronger condition that faithfulness. We introduce this notion in the following definition.

\begin{definition}
An action $H\curvearrowright X$ is called \textit{highly faithful} if, for every finite subset $F\subset X$, for any $n\geq 1$, for any non-empty subsets $S_1,\ldots,S_n\subset X$ such that $X\setminus F\subset \cup_{k=1}^n S_k$ there exists $1\leq k\leq n$ satisfying the following property:
$$\text{if }h\in H\text{ is such that }hx=x\text{ for all }x\in S_k\text{ then }h=1.$$
\end{definition} 

\vspace{0.2cm}

\noindent Given a set $X$ and an integer $k\geq 2$, let $X^{(k)}$ be the \textit{complement of the large diagonal in} $X^k$:
$$X^{(k)}=\{(x_1,\ldots,x_k)\in X^k\,:\,x_i\neq x_j\,\,\text{for all}\,\,i\neq j\}.$$
We give some caracterisations of core-freeness in the next lemma.

\begin{lemma}
Let $\Sigma< H$ be a non-trivial subgroup of an infinite group $H$. The following are equivalent.
\begin{enumerate}
\item $\Sigma$ is highly core-free.
\item $\forall n\geq 1$, $\forall x_1,\ldots,x_n\in H\setminus\{1\}$, $\forall F\subset H$ finite subset, the set
$$H_{F,\bar{x}}=\{h\in H\,:\,hx_i\notin\Sigma F\,\forall i\,\,\text{and}\,\,hx_ih^{-1}\notin\Sigma\,\,\forall i\}\quad\text{is non-empty}.$$
\item $\forall k\geq 2$, $\forall \bar{x}=(x_1,\dots, x_k)\in H^{(k)}$, and for every finite subset $F\subset H$, the set
$$G_{F,\bar{x}}=\{h\in H\,:\,hx_i\notin \Sigma F\,\,\forall i \textrm{ and } hx_ix_j^{-1}h^{-1}\notin \Sigma\,\, \forall i\neq j\}\quad\text{is non-empty}.$$
\item For every free action $H\curvearrowright X$, $\forall k\geq 2$, $\forall \bar{x}=(x_1,\dots, x_k)\in X^{(k)}$ and all $F\subset X$ finite,
$$E_{F,\bar{x}}=\{h\in H\,:\,hx_i\notin \Sigma F\,\,\forall i \textrm{ and } \Sigma hx_i\cap\Sigma hx_j=\emptyset\,\, \forall i\neq j\}\quad\text{is non-empty}.$$
\item The action $H\curvearrowright H/\Sigma$ on left cosets (resp. on right cosets) is highly faithful.
\end{enumerate}
\end{lemma}

\begin{proof}
$1\implies 2.$ Let $n\geq 1$, $x_1,\ldots,x_n\in H\setminus\{1\}$ and $F\subset H$ a finite subset such that $H_{F,\bar{x}}=\emptyset$. Define $F'=\cup_{i=1}^n Fx_i^{-1}$ and, for $k=1,\ldots ,n$, $S_k=\{h\in H\,:\, hx_kh^{-1}\in\Sigma\}$. Let $I=\{k\in\{1,\ldots,n\}\,:\,S_k\neq\emptyset\}$. Since $H_{F,\bar{x}}=\emptyset$ we have $H\setminus\Sigma F'\subset\cup_{k\in I} S_k$ (and $I$ is non-empty since $\Sigma$ is highly core-free). Since, for all $k\in I$, $1\neq x_k\in\rm{Core}_{S_k}(\Sigma)$, $\Sigma$ is not highly core-free.

\vspace{0.2cm}

\noindent $2\implies 3.$ Let $k\geq 2$, $\bar{x}=(x_1,\dots, x_k)\in H^{(k)}$ and $F\subset H$ finite. Consider the collection of non-trivial elements $y_{ij}=x_ix_j^{-1}\in H$  for $i\neq j$ and the finite set $F'=\cup_{i=1}^kFx_i^{-1}$. It is easy to see that $H_{F',\bar{y}}\subset G_{F,\bar{x}}$.

\vspace{0.2cm}

\noindent $3\implies 4.$ Since $H\curvearrowright X$ is free, we may assume that $X=H\times I$, where $I\subset\N$ and the action is given by $h(g,n)=(hg,n)$ for $h\in H$ and $(g,n)\in X$. Let $k\geq 2$, $\bar{x}=(x_1,\ldots,x_k)\in X^{(k)}$ and $F\subset X$, a finite subset. Write $x_i=(g_i,n_i)$ and $\Sigma F=\sqcup_{i=1}^l\Sigma(t_i,m_i)$. Consider the set $F'=\{t_i\,:\,1\leq i\leq l\}\subset H$ and let $m$ be the size of the set $\{g_i\,:\,1\leq i\leq k\}$. If $m\geq 2$ let $\bar{y}=(y_1,\ldots y_m)\in H^{(m)}$ such that $\{g_i\,:\,1\leq i\leq k\}=\{y_1,\ldots,y_m\}$. It is easy to see that $G_{F',\bar{y}}\subset E_{F,\bar{x}}$. If $m=1$, all the $g_i$ are equal to $y\in H$. Since the $x_i$'s are pairwise distinct, the $n_i$'s must be pairwise distinct. Hence, $E_{F,\bar{x}}=\{h\in H\,:\,hx_i\notin\Sigma F\,\,,\forall i\}$. Taking any $y'\in H$ with $y'\neq y$ we have $G_{F',\bar{y}}\subset E_{F,\bar{x}}$, where $\bar{y}=(y,y')$.

\vspace{0.2cm}

\noindent $4\implies 3.$ Applying $4$ to the free action $H\curvearrowright H$ we conclude that $3$ holds.

\vspace{0.2cm}

\noindent $3\implies 2.$ Let $n\geq 1$, $x_1,\ldots,x_n\in H\setminus\{1\}$ and $F\subset H$ finite. To show that $H_{F,\bar{x}}\neq\emptyset$, we may suppose that the $x_i$ are pairwise distinct. Define $y_i=x_i$ for $1\leq i\leq n$. Define $k=n+1$ and $y_k=1$. Then $\bar{y}=(y_1,\ldots,y_k)\in H^{(k)}$ and it is easy to see that $G_{F,\bar{y}}\subset H_{F,\bar{x}}$.

\vspace{0.2cm}

\noindent $2\implies 1.$ Let $F\subset H$ a finite subset, $S_1,\ldots, S_n\subset H$ non-empty such that $H\setminus \Sigma F\subset\cup_{k=1}^n S_k$ and Core$_{S_k}(\Sigma)$ contains a non-trivial element $x_k$ for every $k$. Let $\bar{x}=(x_1,\ldots, x_n)$ and $F'=Fx_1$. If $h\in H$ is such that $hx_i\notin\Sigma F'$ for all $i$ then, $hx_1\notin\Sigma F'$. Hence, $h\notin\Sigma F$. By hypothesis, there exists $1\leq k\leq n$ such that $h\in S_k$. Since $x_k\in\cap_{h\in S_k}h^{-1}\Sigma h$ we find that $hx_kh^{-1}\in\Sigma$. Hence, $H_{F',\bar{x}}=\emptyset$.

\vspace{0.2cm}

\noindent $1\implies 5.$ Let $F\subset H/\Sigma$ be a finite subset, $n\geq 1$, and $S_1,\ldots, S_n\subset H/\Sigma$ non-empty subsets such that $(H/\Sigma)\setminus F\subset\cup_{k=1}^n S_k$. Write $F=(\Sigma F')^{-1}$, where $F'\subset H$ is finite and $S_k=(S_k')^{-1}\Sigma$, where $S_k'\subset H$ is non-empty. It is easy to check that $H\setminus\Sigma F'\subset\cup_{k=1}^n\Sigma S_k'$ hence, there exists $1\leq k\leq n$ such that Core$_{\Sigma S'_k}(\Sigma)=\text{Core}_{S'_k}(\Sigma)=\{1\}$. Let $h\in H$ such that $hx=x$ for all $x\in S_k$. Then for every $y\in S'_k$, we have $hy^{-1}\Sigma=y^{-1}\Sigma$. So $h\in \cap_{y\in S'_k} y^{-1}\Sigma y=\text{Core}_{S'_k}(\Sigma)=\{1\}$.

\vspace{0.2cm}

\noindent $5\implies 1.$ Let $F\subset H$ be a finite subset, $n\geq 1$, and $S_1,\ldots, S_n\subset H$ such that $H\setminus\Sigma F\subset\cup_{k=1}^n S_k$. Define the finite subset $F'=F^{-1}\Sigma\subset H/\Sigma$ and the non-empty subsets $S_k'=S_k^{-1}\Sigma\subset H/\Sigma$. It is easy to see that $H/\Sigma\setminus F'\subset\cup_{k=1}^n S'_k$. By hypothesis there exists $k$ such that $hx=x\, \forall x\in S'_k$, for the action of $H$ on the left cosets $H/\Sigma$, implies $h=1$. Thus Core$_{S_k}(\Sigma)=\{1\}$.
\end{proof}
\begin{example}
The highly core-free condition is stricly stronger than core-freeness, even for infinite groups. Indeed, let $H=S_{\infty}$ be the group of finitely supported permutations of $\N$ and $\Sigma=\text{Stab}_{S_{\infty}}(0)<H$ be the stabilizer of $0\in\N$. Since $S_{\infty}\curvearrowright\N$ is transitive and faithful, $\Sigma$ is core-free. However, $\Sigma $ is not highly core-free since the action $S_{\infty}\curvearrowright\N$ is not highly faithful: it suffices to write $\N=\{0,1\}\cup\{k\in\N\,:\,k\geq 2\}$ and to use directly the definition.
\end{example}

\noindent Motivated by the preceding lemma, we give the following definition.

\begin{definition}
Let $H\curvearrowright X$ be an action of a group $H$ on an infinite set $X$ and $\Sigma< H$ a subgroup. We say that $\Sigma$ \textit{is highly core-free with respect to the action} $H\curvearrowright X$ if, $\forall k\geq 2$, $\forall \bar{x}=(x_1,\dots, x_k)\in X^{(k)}$, and for every finite subset $F\subset X$, the set
$$E_{F,\bar{x}}=\{h\in H\,:\,hx_i\notin \Sigma F\,\,\forall i \textrm{ and } \Sigma hx_i\cap\Sigma hx_j=\emptyset\,\, \forall i\neq j\}\quad\text{is non-empty}.$$
\end{definition}

\noindent Notice that if $\Sigma$ is highly core-free with respect to the action $H\curvearrowright X$, then the action $H\curvearrowright X$ has no finite orbits. Notice also that if $\Sigma$ is highly core-free with respect to the action $H\curvearrowright X$, then the action $\Sigma\curvearrowright X$ has infinitely many orbits.

\begin{remark}
Let $H$ be a group and $\Sigma< H$ a subgroup. If there exists a highly transitive action $H\curvearrowright X$ such that the action $\Sigma \curvearrowright X$ has infinitely many orbits, then $\Sigma$ is highly core-free w.r.t. $H\curvearrowright X$.
\end{remark}

\begin{proof}
Let $k\geq 2$, $\bar{x}=(x_1,\dots,x_k)\in X^{(k)}$ and $F\subset X$ a finite subset. We want to prove that
$$E_{F,\bar{x}}=\{h\in H\,:\,hx_i\notin \Sigma F\,\,\forall i \textrm{ and } \Sigma hx_i\cap \Sigma hx_j=\emptyset\,\, \forall i\neq j\}\neq\emptyset.$$
Let $Y=\Sigma F\cup(\cup_{j=1}^k  \Sigma x_j)$. Since $\Sigma \curvearrowright X$ has infinitely many orbits there exists an element $(z_1,\dots, z_k)\in X^{(k)}$ such that $\Sigma z_i\subset Y^c$ and $\Sigma z_i\cap \Sigma z_j=\emptyset$ for all $i\neq j$. By high transitivity of $H\curvearrowright X$ there exists $h \in H$ such that $h x_i=z_i$ for every $i$. Then $h x_i=z_i \notin \Sigma F$ for all $i$ and $\Sigma h x_i\cap \Sigma h x_j=\Sigma z_i\cap \Sigma z_j=\emptyset$ for all $i\neq j$ so $h \in E_{F,\bar{x}}$.
\end{proof}

\noindent Let $H$ be a group. For $g\in H$ denote by $Cl_H(g)=\{hgh^{-1}\,:\,h\in H\}$ the conjugacy class of $g$ and by $C_H(g)=\{h\in H\,:\,gh=hg\}$ the centralizer of $g$. Recall that the group $H$ is called icc if for every $g\in H\setminus\{1\}$ the set $Cl_H(g)$ is infinite (or, equivalently, the subgroup $C_H(g)$ has infinite index). We call a subgroup $\Sigma<H$ \textit{icc relative to $H$} if, for every $\sigma\in\Sigma\setminus\{1\}$,  the set $Cl_H(\sigma)$ is infinite.

\begin{lemma}\label{hcf}
Let $\Sigma< H$ be a subgroup such that
\begin{itemize}
\item $\Sigma$ has infinite index.
\item $\Sigma$ is icc relative to $H$.
\item For any $h\in H\setminus\{1\}$ the set $\Sigma\cap Cl_H(h)$ is finite.
\end{itemize}
Then, for all $n\geq 1$, for all $x_1,\dots, x_n\in H\setminus\{1\}$ and for every finite subset $F\subset H$, the set
$$H_{F,\bar{x}}=\{h\in H\,:\,hx_i\notin \Sigma F\,\,\forall i \textrm{ and } hx_ih^{-1}\notin \Sigma \,\,\forall i\}$$
contains infinitely many $\Sigma$-classes. In particular, $\Sigma$ is highly core-free.
\end{lemma}

\begin{proof}
We suppose that $\Sigma$ has infinite index, the set $\Sigma\cap Cl_H(h)$ is finite, $\forall h\in H\setminus\{1\}$, and there exist $n\geq 1$, $ x_1,\dots, x_n\in H\setminus\{1\}$ and $F\subset H$ finite such that $H_{F,\bar{x}}$ contains finitely many $\Sigma$-classes. We shall show that there exists $\sigma\in \Sigma\setminus\{1\}$ such that $C_H(\sigma)$ has finite index.

\vspace{0.2cm}

\noindent Define $Y=H_{F,\bar{x}}\sqcup \{h\in H\,:\,\exists i\,\,hx_i\in\Sigma F\}$ and $H_{i,\sigma}=\{h\in H: hx_ih^{-1}=\sigma\}$, for $1\leq i\leq n$ and $\sigma\in\Sigma$. Since $\{h\in H\,:\,\exists i\,\,hx_i\in\Sigma F\}=\cup_i\Sigma Fx_i^{-1}$, $Y$ is a finite union of $\Sigma$-classes. Also, note that $H\setminus Y\subset \bigcup_{1\leq i\leq n,\sigma\in\Sigma} H_{i,\sigma}$.
 
\vspace{0.2cm}

\noindent For $i=1,\ldots,n$ and $\sigma\in \Sigma$, if $H_{i,\sigma}\neq \emptyset$, it is a coset of the form $C_H(\sigma)g_{i,\sigma}$ for $g_{i,\sigma}\in H_{i,\sigma}$. Indeed, if $g\in H_{i,\sigma}$, we have $gg_{i,\sigma}^{-1}\sigma(gg_{i,\sigma}^{-1})^{-1}=gx_ig^{-1}=\sigma$, hence $gg_{i,\sigma}^{-1}\in C_H(\sigma)$. Conversely, if $h\in C_H(\sigma)$, we have $(hg_{i,\sigma})x_i(hg_{i,\sigma})^{-1}=\sigma$, so $hg_{i,\sigma}\in H_{i,\sigma}$.

\vspace{0.2cm}

\noindent Since, for all $1\leq i\leq n$, the set $\{\sigma\in\Sigma\,:\,H_{i,\sigma}\neq\emptyset\}$ is finite as it is a subset of $\Sigma\cap Cl_H(x_i)$ it follows that $H\setminus Y$ is covered by a finite union of the cosets $C_H(\sigma)g_{i,\sigma}$ for $1\leq i\leq n$ and $\sigma\in\Sigma$ such that $H_{i,\sigma}\neq\emptyset$. For those cosets $C_H(\sigma)g_{i,\sigma}$ one has $\sigma\neq 1$ since $H_{i,\sigma}\neq\emptyset$ implies that $\sigma$ is conjugated to some non-trivial $x_i$. Since $Y$ is a finite union of $\Sigma$-classes, $H$ itself is a finite union of cosets of subgroups either equal to $\Sigma$ or of the form $C_H(\sigma)$ where $\sigma\in\Sigma\setminus\{1\}$. Thus by \cite[Lemma 4.1]{NeumannCovered}, and since $\Sigma$ is supposed to have infinite index, there exists $\sigma\in \Sigma\setminus\{1\}$ such that $C_H(\sigma)$ has finite index.
\end{proof}

\begin{example}\label{Exhcf}
The easiest example of groups satisfying the hypothesis of Lemma \ref{hcf} is a finite relatively icc subgroup $\Sigma$ of $H$. In particular, a finite subgroup of an icc group. Other examples are given below.

\vspace{0.2cm}

\noindent Observe that, if $\Sigma< H$ is malnormal, then $C_H(\sigma)\subset\Sigma$ for all $\sigma\in\Sigma\setminus\{1\}$. Hence, if $\Sigma$ is malnormal and has infinite index then $\Sigma$ is relatively icc.

\vspace{0.2cm}

\noindent Suppose that $\Sigma<H$ is a malnormal subgroup then, for all $h\in H\setminus\{1\}$ with the property that $Cl_H(h)\cap\Sigma\neq\emptyset$, there exists $\sigma_0\in\Sigma\setminus\{1\}$ such that $Cl_H(h)\cap\Sigma= Cl_{\Sigma}(\sigma_0)$. Indeed, let $h\neq 1$ and $\sigma_0=ghg^{-1} \in \Sigma$ for some $g\in H$. If $\sigma=tht^{-1}\in Cl_H(h)\cap\Sigma$ then $\sigma\neq 1$ and, $\sigma= tg^{-1}\sigma_0 g t^{-1}=tg^{-1}\sigma_0 (tg^{-1})^{-1}\in tg^{-1}\Sigma (tg^{-1})^{-1}$. So $\sigma\in\Sigma\cap tg^{-1}\Sigma gt^{-1}$ and thus $tg^{-1}\in\Sigma $ by malnormality of $\Sigma$. Hence $\sigma \in Cl_{\Sigma}(\sigma_0)$ and this proves $Cl_H(h)\cap\Sigma\subset Cl_{\Sigma}(\sigma_0)$. The other inclusion is obvious.

\vspace{0.2cm}

\noindent By the preceding remarks any infinite index and malnormal subgroup $\Sigma< H$ such that $Cl_{\Sigma}(\sigma)$ is finite for all $\sigma\in\Sigma$ satisfy the hypothesis of Lemma \ref{hcf}. Here are some particular examples.
\begin{itemize}
\item A finite malnormal subgroup $\Sigma$ of an infinite group $H$.
\item $H=\Sigma* G$, $G$ and $\Sigma$ are non-trivial and $\Sigma$ abelian. More generally, an abelian malnormal subgroup $\Sigma$ of infinite index of an infinite group $H$.
\item Let $K$ be an infinite (commutative) field. Let $H=K^*\ltimes K$ and $\Sigma=K^*< H$. It is shown in \cite{HW11} that $\Sigma$ is malnormal in $H$.
\item Let $H=\langle a,b \rangle\simeq \mathbf{F}_2$ be a free group on 2 generators $a$ and $b$. The infinite cyclic subgroup $\Sigma$ generated by any primitive element (e.g. $a^kba^{-l}b^{-1}$ with $k\neq 0 \neq l\in\Z $) is malnormal (see Example 7.A. in \cite{HW11} and \cite{BMR}).
\end{itemize}
\end{example}

%%%%%%%%%%%%%%%%%%%%%%%%%%%%%%%%%%%
\section{The case of an HNN-extension}
%%%%%%%%%%%%%%%%%%%%%%%%%%%%%%%%%%%%%

\noindent Let $H$ be a countable infinite group, $\Sigma<H$ a subgroup and $\theta\,:\,\Sigma\rightarrow H$ an injective group homomorphism. Let $\Gamma={\rm HNN}(H,\Sigma,\theta)=\langle H,t\,|\,\theta(\sigma)=t\sigma t^{-1}\,\forall\sigma\in\Sigma\rangle$ and, for $\epsilon\in\{-1,1\}$,
$$\Sigma_{\epsilon}=\left\{\begin{array}{lcl}
\Sigma &\text{if}&\epsilon=1,\\
\theta(\Sigma) &\text{if}&\epsilon=-1.
\end{array}\right.$$
Let $H\curvearrowright X$ be an action of $H$ on an infinite countable set $X$. Let $S(X)$ be the Polish group of bijections of $X$. Although the action is not supposed to be faithful we write, in order to simplify the notations, the same symbol $h\in S(X)$ for the action of $h\in H$ on the set $X$.
\vspace{0.2cm}

\noindent Define
$$Z=\{w\in S(X)\,:\,\theta(\sigma)=w\sigma w^{-1}\,\,\,\forall\sigma\in\Sigma\}.$$
Then $Z$ is a closed subset of $S(X)$. Suppose moreover that the actions $\Sigma_{\epsilon}\curvearrowright X$ are free for $\epsilon\in\{-1,1\}$ and have infinitely many orbits (this last assumption is automatically satisfied if $\Sigma_{\epsilon}$ is h.c.f. w.r.t. $H\curvearrowright X$ for all $\epsilon\in\{-1,1\}$). With these assumptions, $Z$ is non-empty.

\vspace{0.2cm}

\noindent For all $w\in Z$ there exists a unique group homomorphism $\pi_w\,:\,\Gamma\rightarrow S(X)$ such that $\pi_w(t)=w$ and $\pi_w(h)=h$ for all $h\in H$.

\begin{lemma}\label{HNNHT}
If for $\epsilon\in\{-1,1\}$, $\Sigma_{\epsilon}$ is highly core-free w.r.t. $H\curvearrowright X$ then the set\\
 $O=\{w\in Z\,:\,\pi_w\,\,\text{is highly transitive}\}$ is a dense $G_{\delta}$ in $Z$.
\end{lemma}

\begin{proof}
Since a $2$-transitive action is transitive we have $O=\cap_{n\geq 2}\cap_{\bar{x},\bar{y}\in X^{(n)}} O_{\bar{x},\bar{y}}$ where, for $\bar{x}=(x_1,\ldots x_n),\bar{y}=(y_1,\ldots, y_n)\in X^{(n)}$, 
$$O_{\bar{x},\bar{y}}=\{w\in Z\,:\,\exists g\in\Gamma,\,\,\pi_w(g)x_k=y_k\,\,\forall 1\leq k\leq n\}.$$
It is easy to see that $O_{\bar{x},\bar{y}}$ is open in $Z$. Let us show that $O_{\bar{x},\bar{y}}$ is dense in $Z$. Let $w\in Z$ and $F\subset X$ be a finite subset.
We claim that there exists $g,h\in H$ such that $hx_k\notin\Sigma F$, $g^{-1}y_k\notin w(\Sigma F)$ for all $1\leq k\leq n$ and the sets $\Sigma hx_k$, $\Sigma w^{-1} g^{-1}y_k$ for $1\leq k\leq n$ are pairwise disjoint. Indeed, since $\theta(\Sigma)$ is highly core-free w.r.t. $H \curvearrowright X$, the set
$$E_{wF,\bar{y}}=\{h\in H\,:\,hy_i\notin \theta(\Sigma)wF\,\,\forall i \textrm{ and } \theta(\Sigma)hy_i\cap \theta(\Sigma)hy_j=\emptyset\,\, \forall i\neq j\}$$
is not empty. Take $g^{-1}\in E_{wF,\bar{y}}$ and let $F':=F\cup (\cup_{k=1}^n w^{-1}g^{-1}y_k)$. Since $\Sigma$ is highly core-free w.r.t. $H \curvearrowright X$, the set $E_{F',\bar{x}}=\{h\in H\,:\,hx_i\notin \Sigma F'\,\,\forall i \textrm{ and } \Sigma hx_i\cap \Sigma hx_j=\emptyset\,\, \forall i\neq j\}$ is not empty and we take $h\in E_{F',\bar{x}}$. It is clear that $g$, $h$ satisfy the claimed properties.

\vspace{0.2cm}

\noindent Define $Y=\sqcup_{k=1}^n\Sigma hx_k\sqcup\Sigma w^{-1}g^{-1}y_k$. Then $F\subset Y^c$ and $w(Y)=\sqcup_{k=1}^n\theta(\Sigma) whx_k\sqcup\theta(\Sigma)g^{-1}y_k$. Since the actions $\Sigma_{\epsilon}\curvearrowright X$ are free, we can define $\gamma\in S(X)$ by $\gamma|_{Y^c}=w|_{Y^c}$ and,
$$\gamma(\sigma hx_k)=\theta(\sigma)g^{-1}y_k,\quad\gamma(\sigma w^{-1}g^{-1}y_k)=\theta(\sigma) whx_k\,\,\,\sigma\in\Sigma,\,\,1\leq k\leq n.$$
By construction $\gamma\in Z$ and $\gamma|_F=w|_F$. Moreover, $\pi_{\gamma}(gth)x_k=g\gamma hx_k=gg^{-1}y_k=y_k$ for all $1\leq k\leq n$.
\end{proof}

\begin{theorem}\label{ThmHNN}
Let $H$ be a infinite countable group, $\Sigma< H$ a subgroup and $\theta\,:\,\Sigma\rightarrow H$ an injective group homomorphism. Suppose that, for $\epsilon\in\{-1,1\}$, $\Sigma_{\epsilon}$ is highly core-free in $H$. Then $\Gamma$ admits a highly transitive and faithful action on an infinite countable set. Moreover, if $H$ is amenable and $\Sigma$ is finite, this action can be chosen to be amenable.
\end{theorem}

\begin{proof}
 Consider the free action $\Gamma\curvearrowright X$ with $X=\Gamma\times\N$ given by $g(x,n)=(gx,n)$ for $g\in \Gamma$ and $(x,n)\in X$. Since $\Sigma_{\epsilon}$ is highly core-free in $H$, it is also highly core-free w.r.t. $H\curvearrowright X$. Hence, we can apply Lemma \ref{HNNHT} to the action $H\curvearrowright X$ and it suffices to show that the set $O=\{w\in Z\,:\,\pi_w\,\,\text{is faithful}\}$ is a dense $G_{\delta}$ in $Z$. Writing $O=\cap_{g\in\Gamma\setminus\{1\}}O_g$, where the set $O_g=\{w\in Z\,:\,\pi_w(g)\neq\id\}$ is obviously open, it suffices to show that $O_g$ is dense. We are going to prove directly that $O$ itself is dense. Write $X=\cup^{\uparrow}X_n$ where $X_n=\{(x,k)\in X\,:\,k\leq n\}$ is infinite and globally invariant under $\Gamma$. Let $w\in Z$ and $F\subset X$ a finite subset. Let $N\in\N$ large enough such that $\Sigma F\cup w(\Sigma F)\subset X_N$. Since $\Sigma_{\epsilon}$ has infinite index in $\Gamma$, the set $X_N\setminus\Sigma F$ (resp. $X_N\setminus w(\Sigma F)$) has infinitely many $\Sigma$-orbits (resp. $\theta(\Sigma)$-orbits) and is globally invariant under $\Sigma$ (resp. $\theta(\Sigma)$). Hence, there exists a bijection $\gamma_0\,:\,X_N\setminus\Sigma F\rightarrow X_N\setminus w(\Sigma F)$ satisfying $\gamma_0\sigma=\theta(\sigma)\gamma_0$ for all $\sigma\in\Sigma$. Define $\gamma\in S(X)$ by $\gamma|_{\Sigma F}=w|_{\Sigma F}$, $\gamma|_{X_N\setminus\Sigma F}=\gamma_0$ and $\gamma|_{X_N^c}=t|_{X_N^c}$. By construction, $\gamma\in Z$ and $\gamma|_F=w|_F$. Moreover, since $\pi_{\gamma}(g)(x,n)=(gx,n)$ for all $n>N$ and since $\Gamma\curvearrowright X$ is faithful, it follows that $\pi_{\gamma}$ is faithful.

\vspace{0.2cm}

\noindent If $H$ is amenable and $\Sigma$ is finite, let $(D_n)$ be a F{\o}lner sequence in $H$ such that $|D_n|\rightarrow\infty$. Then $C_n=D_n\times\{0\}$ is a F{\o}lner sequence for $H\curvearrowright X$ such that $|C_n|\rightarrow\infty$. We may apply \cite[Lemma 3.2]{Fi12} to conclude.
\end{proof}

%%%%%%%%%%%%%%%%%%%%%%%%%%%
\section{The case of an amalgamated free product}
%%%%%%%%%%%%%%%%%%%%%%%%%%%%

\noindent Let $\Gamma_1$, $\Gamma_2$ be two infinite countable groups and $\Sigma$ a common subgroup of $\Gamma_1$, $\Gamma_2$.  Let $\Gamma=\Gamma_1*_{\Sigma}\Gamma_2$ be the amalgamated free product. Suppose that, for $i=1,2$, there is an action $\Gamma_i\curvearrowright X$ on an infinite countable set $X$ such that the two actions of $\Sigma$ on $X$ are free.

\vspace{0.2cm}

\noindent Observe that if the two actions of $\Sigma$ on $X$ have infinitely many orbits (which is automatic when $\Sigma$ is h.c.f. w.r.t. $\Gamma_i\curvearrowright X$ for $i=1,2$) then we may suppose, up to conjugating the action of $\Gamma_2$ by an element in $S(X)$, that the two actions of $\Sigma$ on $X$ coincide. Hence we do suppose that these two actions actually coincide.

\vspace{0.2cm}

\noindent As before, although the actions $\Gamma_i\curvearrowright X$ are not supposed to be faithful, we use the same symbol $g\in S(X)$ to denote the action of the element $g\in\Gamma_i$ on the set $X$ for $i=1,2$. Define
$$Z=\{w\in S(X)\,:\,w\sigma=\sigma w\,\,\forall\sigma\in\Sigma\}.$$
It is easy to check that $Z$ is a closed subgroup of $S(X)$.

\vspace{0.2cm}

\noindent For all $w\in Z$ there exists a unique group homomorphism $\pi_w\,:\,\Gamma\rightarrow S(X)$ such that $\pi_w(g)=g$ and $\pi_w(h)=w^{-1}hw$ for all $g\in\Gamma_1$, $h\in \Gamma_2$.

\begin{lemma}\label{FPHT}
If for $i=1,2$, $\Sigma$ is highly core-free w.r.t. $\Gamma_i\curvearrowright X$, then the set\\
$O=\{w\in Z\,:\,\pi_w\,\,\text{is highly transitive}\}$ is a dense $G_{\delta}$ in $Z$.
\end{lemma}

\begin{proof}
Write $O=\cap_{n\geq 2}\cap_{\bar{x},\bar{y}\in X^{(n)}} O_{\bar{x},\bar{y}}$ where, for $\bar{x}=(x_1,\ldots x_n),\bar{y}=(y_1,\ldots, y_n)\in X^{(n)}$,
$$O_{\bar{x},\bar{y}}=\{w\in Z\,:\,\exists g\in\Gamma,\,\,\pi_w(g)x_k=y_k\,\,\forall 1\leq k\leq n\}.$$
Since $O_{\bar{x},\bar{y}}$ is obviously open in $Z$, it suffices to show that it is dense in $Z$. Let $w\in Z$ and $F\subset X$ be a finite subset. We first prove the following claim.

\vspace{0.2cm}

\noindent \textbf{Claim.} There exists $g_1,g_2\in\Gamma_1$, $h\in\Gamma_2$ and $z_1,\ldots,z_n\in X$ such that
\begin{itemize}
\item For all $1\leq k\leq n$, $\Sigma g_1 x_k\cap F=\emptyset$, $\Sigma g_2^{-1} y_k\cap F=\emptyset$, $\Sigma w^{-1}z_k\cap F=\emptyset$, $\Sigma w^{-1}hz_k\cap F=\emptyset$.
\item The sets $\Sigma g_1 x_k$, $\Sigma g_2^{-1} y_k$, $\Sigma w^{-1}z_k$ and $\Sigma w^{-1}hz_k$ are pairwise disjoint for $1\leq k\leq n$.
\end{itemize}

\textit{Proof of the claim.} Since $\Sigma$ is highly core-free w.r.t. $\Gamma_1 \curvearrowright X$, the set
$$E_{F,\bar{x}}=\{h\in \Gamma_1\,:\,hx_i\notin \Sigma F\,\,\forall i \textrm{ and } \Sigma hx_i\cap \Sigma h x_j=\emptyset\,\, \forall i\neq j\}$$
is not empty. Take $g_1\in E_{F,\bar{x}}$. Let $F':=F\cup\{g_1x_i\,:\,1\leq i\leq n\}$. Since the set
$$E_{F',\bar{y}}=\{h\in \Gamma_1\,:\,hy_i\notin \Sigma F'\,\,\forall i \textrm{ and } \Sigma hy_i\cap \Sigma h y_j=\emptyset\,\, \forall i\neq j\}$$
is not empty, we can take $g_2^{-1}\in E_{F',\bar{y}}$.
Let
$$
Y=w\big(\Sigma F\cup (\cup_{i=1}^n\Sigma g_1 x_i)\cup (\cup_{i=1}^n\Sigma g_2^{-1}y_i)\big).
$$
Since $\Sigma\curvearrowright X$ has infinitely many orbits, we can find $n$ distinct $\Sigma$-classes $\Sigma z_1,\ldots,\Sigma z_n$ in the complement of $Y$.
Let $F'':=\cup_{i=1}^n \Sigma z_i \cup Y$. Since $\Sigma$ is highly core-free w.r.t. $\Gamma_2 \curvearrowright X$, the set
$$E_{F'',\bar{z}}=\{h\in \Gamma_2\,:\,hz_i\notin \Sigma F''\,\,\forall i \textrm{ and } \Sigma hz_i\cap \Sigma h z_j=\emptyset\,\, \forall i\neq j\}
$$
is not empty. So there is $h\in E_{F'',\bar{z}}$ and this proves the claim.

\vspace{0.2cm}

\noindent \textit{End of the proof of the lemma.} Define $Y'=\bigsqcup_{k=1}^n\Sigma g_1x_k\sqcup\Sigma g_2^{-1}y_k\sqcup\Sigma w^{-1} z_k\sqcup\Sigma w^{-1}hz_k$. One has $F\subset Y'^c$ and $w(Y')=\bigsqcup_{k=1}^n\Sigma wg_1 x_k\sqcup\Sigma wg_2^{-1}y_k\sqcup\Sigma z_k\sqcup\Sigma hz_k$. Define $\gamma\in S(X)$ by $\gamma|_{Y'^c}=w|_{Y'^c}$ and, for all $\sigma\in\Sigma$ and all $1\leq k\leq n$,
$$\gamma(\sigma g_1x_k)=\sigma z_k,\,\,\,\gamma(\sigma g_2^{-1}y_k)=\sigma hz_k,\,\,\,\gamma(\sigma w^{-1}z_k)=\sigma wg_1 x_k,\,\,\,\gamma(\sigma w^{-1}hz_k)=\sigma wg_2^{-1} y_k.$$
By construction, $\gamma\in Z$ and $\gamma|_F=w|_F$. Moreover, with $g=g_2hg_1\in\Gamma$, one has, for all $1\leq k\leq n$,
$$\pi_{\gamma}(g)x_k=g_2\gamma^{-1}h\gamma g_1x_k=g_2\gamma^{-1}hz_k=g_2g_2^{-1}y_k=y_k.$$
\end{proof}

\begin{theorem}\label{ThmFree}
Suppose that $\Gamma_1,\Gamma_2$ are infinite countable groups with a common subgroup $\Sigma$. If $\Sigma$ is highly core-free in both $\Gamma_1$ and $\Gamma_2$, then $\Gamma$ admits a faithful and highly transitive action on an infinite countable set. If moreover $\Gamma_1, \Gamma_2$ are amenable and $\Sigma$ is finite then the action can be chosen to be amenable.
\end{theorem}

\begin{proof}
Consider the free action $\Gamma\curvearrowright X$ with $X=\Gamma\times\N$ given by $g(x,n)=(gx,n)$ for $g\in \Gamma$ and $(x,n)\in X$. Since $\Sigma$ is highly core-free in $\Gamma_i$ it is also highly core-free w.r.t. $\Gamma_i\curvearrowright X$ for $i=1,2$. Hence, we can apply Lemma \ref{FPHT} to the actions $\Gamma_i\curvearrowright X$, $i=1,2$, and it suffices to show that the set $O=\{w\in Z\,:\,\pi_w\,\,\text{is faithful}\}$ is a dense $G_{\delta}$ in $Z$. Writing $O=\cap_{g\in\Gamma\setminus\{1\}}O_g$, where $O_g=\{w\in Z\,:\,\pi_w(g)\neq\id\}$ is obviously open, it suffices to show that $O_g$ is dense. Let us prove that $O$ itself is dense. Write $X=\cup^{\uparrow}X_n$ where $X_n=\{(x,k)\in X\,:\,k\leq n\}$ is infinite and globally invariant under $\Gamma$. Let $w\in Z$ and $F\subset X$ a finite subset. Let $N\in\N$ large enough such that $\Sigma F\cup w(\Sigma F)\subset X_N$. Since $\Sigma$ has infinite index in $\Gamma$, the sets $X_N\setminus\Sigma F$ and $X_N\setminus w(\Sigma F)$ have infinitely many $\Sigma$-orbits and are globally invariant under $\Sigma$. Hence, there exists a bijection $\gamma_0\,:\,X_N\setminus\Sigma F\rightarrow X_N\setminus w(\Sigma F)$ satisfying $\gamma_0\sigma=\sigma\gamma_0$ for all $\sigma\in\Sigma$. Define $\gamma\in S(X)$ by $\gamma|_{\Sigma F}=w|_{\Sigma F}$, $\gamma|_{X_N\setminus\Sigma F}=\gamma_0$ and $\gamma|_{X_N^c}=\id|_{X_N^c}$. By construction, $\gamma\in Z$ and $\gamma|_F=w|_F$. Moreover, since $\pi_{\gamma}(g)(x,n)=(gx,n)$ for all $n>N$ and since $\Gamma\curvearrowright X$ is faithful, it follows that $\pi_{\gamma}$ is faithful.

\vspace{0.2cm}

\noindent If $\Gamma_i$ is amenable for $i=1,2$ then there exists F{\o}lner sequences $(C'_n)$ in $\Gamma_1$ and $(D'_n)$ in $\Gamma_2$ such that  $|C'_n|=|D'_n|\rightarrow\infty$. Since $\Sigma$ is finite, we may apply \cite[Lemma 4.2]{Fi12} to the F{\o}lner sequences $C_n=C'_n\times\{0\}$ and $D_n=D'_n\times\{0\}$ for $\Gamma_1\curvearrowright X$ and $\Gamma_2\curvearrowright X$ respectively to conclude the proof.
\end{proof}

%%%%%%%%%%%%%%%%%%%%%%%%%%
\section{Groups acting on trees}
%%%%%%%%%%%%%%%%%%%%%%%%%%%

\noindent In this section we prove Theorem A.

\vspace{0.2cm}

\noindent Let $\Gamma$ be a group acting without inversion on a non-trivial tree. By \cite{Se83}, the quotient graph $\mathcal{G}$ can be equipped with the structure of a graph of groups $(\mathcal{G},\{\Gamma_p\}_{p\in\VG},\{\Sigma_e\}_{e\in\EG})$ where each $\Sigma_e=\Sigma_{\overline{e}}$ is isomorphic to an edge stabilizer and each $\Gamma_p$ is isomorphic to a vertex stabilizer and such that $\Gamma$ is isomorphic to the fundamental group $\pi_1(\Gamma, \mathcal{G})$ of this graph of groups i.e., given a fixed maximal subtree $\mathcal{T}\subset\mathcal{G}$, the group $\Gamma$ is generated by the groups $\Gamma_p$ for $p\in\VG$ and the edges $e\in\EG$ with the relations
$$\overline{e}=e^{-1},\quad s_{e}(x)=er_{e}(x)e^{-1}\,\,,\,\forall x\in\Sigma_e\quad\text{and}\quad e=1\,\,\,\,\forall e\in {\rm E}(\mathcal{T}),$$
where $s_e\,:\,\Sigma_e\rightarrow \Gamma_{s(e)}$ and $r_e=s_{\overline{e}}\,:\,\Sigma_e\rightarrow\Gamma_{r(e)}$ are respectively the source and range group monomomorphisms. Using Lemma \ref{hcf} one sees that Theorem A is a straightforward Corollary of the following result.

\begin{theorem}\label{Main}
If $\Gamma_p$ is infinite for all $p\in\VG$ and $s_e(\Sigma_e)$ is highly core-free in  $\Gamma_{s(e)}$ for all $e\in\EG$ then $\Gamma$ admits a faithful and highly transitive action on an infinite countable set.
\end{theorem}

\begin{proof}

Let $e_0$ be one edge of $\mathcal{G}$ and $\mathcal{G}'$ be the graph obtained from $\mathcal{G}$ by removing the edges $e_0$ and $\overline{e_0}$.

\vspace{0.2cm}

\noindent \textbf{Case 1: $\mathcal{G}'$ is connected.} It follows from Bass-Serre theory that $\Gamma={\rm HNN}(H,\Sigma,\theta)$ where $H$ is fundamental group of our graph of groups restricted to $\mathcal{G}'$, $\Sigma=r_{e_0}(\Sigma_{e_0})<H$ is a subgroup and $\theta\,:\,\Sigma\rightarrow H$ is given by $\theta=s_{e_0}\circ r_{e_0}^{-1}$. By hypothesis $H$ is infinite and, since $\Sigma<\Gamma_{r(e_0)}$ (resp. $\theta(\Sigma)<\Gamma_{s(e_0)}$) is a highly core-free subgroup, $\Sigma<H$ (resp. $\theta(\Sigma)<H$) is also a highly core-free subgroup. Thus we may apply Theorem \ref{ThmHNN} to conclude that $\Gamma$ admits a faithful and highly transitive action.

\vspace{0.2cm}

\noindent \textbf{Case 2: $\mathcal{G}'$ is not connected.} Let $\mathcal{G}_1$ and $\mathcal{G}_2$ be the two connected components of $\mathcal{G}'$ such that $s(e_0)\in{\rm V}(\mathcal{G}_1)$ and $r(e_0)\in{\rm V}(\mathcal{G}_2)$. Bass-Serre theory implies that $\Gamma=\Gamma_1*_{\Sigma{e_0}}\Gamma_2$, where $\Gamma_i$ is the fundamental group of our graph of groups restricted to $\mathcal{G}_i$, $i=1,2$, and $\Sigma_{e_0}$ is viewed as a highly core-free subgroup of $\Gamma_1$ via the map $s_{e_0}$ and as a highly core-free subgroup of $\Gamma_2$ via the map $r_{e_0}$ since $s_{e_0}(\Sigma_{e_0})$ is highly core-free in $\Gamma_{s(e_0)}$ and $r_{e_0}(\Sigma_{e_0})$ is highly core-free in $\Gamma_{r(e_0)}$ by hypothesis. Since $\Gamma_1$ and $\Gamma_2$ are infinite, we may apply Theorem \ref{ThmFree} to conclude that $\Gamma$ admits a faithful and highly transitive action.
\end{proof}

%%%%%%%%%%%%%%%%%%%%%%%%%%
\section{Links with former results}
%%%%%%%%%%%%%%%%%%%%%%%%%%%

\noindent Kitroser \cite{Ki12} proved that surface groups admit a faithful and highly transitive action. We recover this result from Theorem \ref{ThmFree}.

\begin{example}\label{Kitroser}
The fundamental group $\pi_1(\Sigma_g)$ of a closed, orientable surface of genus $g > 1$ admits a faithful and highly transitive action.

\vspace{0.2cm}

\noindent Indeed, as seen in Example \ref{Exhcf}, the subgroup $\Sigma=\langle [a_1,b_1]\rangle $ generated by the commutator $[a_1,b_1]=a_1b_1a_1^{-1}b_1^{-1}$ is highly core-free in $\Gamma_1=\langle a_1,b_1 \rangle\simeq \mathbf{F}_2$. So Theorem \ref{ThmFree} implies that the group $\pi_1(\Sigma_2) = \langle a_1,b_1 \rangle\ast_{\langle c\rangle }\langle a_2,b_2 \rangle$ where $c = [a_1, b_1] = [a_2, b_2]$, admits a faithful and highly transitive action. For $g > 1$, the group $\pi_1(\Sigma_g)$ injects into $\pi_1(\Sigma_2)$ as a subgroup of finite index. Thus by Corollary 1.5 in \cite{MoSt},  the group $\pi_1(\Sigma_g)$ admits a faithful and highly transitive action as well. 
\end{example}

\noindent Chaynikov \cite[Section IV.4]{ChaPhD} proved that non-elementary hyperbolic groups with trivial finite radical admit a faithful and highly transitive action. (In fact, this action has more properties.) Of course, we do not recover this result, since our Theorem \ref{Main} is not applicable to one-ended hyperbolic groups. On the other hand our techniques allow to treat some non-hyperbolic groups. A first example is $\Z^2*\Z^2$, which admits a faithful and highly transitive action by results in \cite{MoSt}. New examples are as follows: let us denote by
$\Z[i]$ is the ring of Gaussian integers (which is isomorphic to $\Z^2$ as a group) and by $\Z[i]^*$ the group of its invertible elements (which is cyclic of order $4$).
\begin{example}
 The following groups admit a faithful and highly transitive action:
 \begin{enumerate}
  \item the group $(\Z[i]^*\ltimes \Z[i])*_{\Z[i]^*}(\Z[i]^*\ltimes \Z[i])$, and
  \item the group ${\rm HNN}(\Z[i]^*\ltimes\Z[i], \Z[i]^*, \theta)$, where $\theta$ maps $\Z[i]^*$ onto one of its conjugates.
 \end{enumerate}
\end{example}

\noindent Indeed, $\Z[i]^*$ is a malnormal subgroup in $\Z[i]^*\ltimes \Z[i]$ and so are all its conjugates (see e.g. \cite[Propositions 1 and 2]{HW11}). Since they are moreover finite, with infinite index, these subgroups satisfy the hypotheses of Lemma \ref{hcf}, so that they are highly core-free. It then suffices to apply Theorem \ref{Main}.
%%%%%%%%%%%%%%%%%%%%%%%%%%%%%%


\begin{thebibliography}{5}

\bibitem[BMR99]{BMR}
G. Baumslag, A. Myasnikov, and V. Remeslennikov, Malnormality
is decidable in free groups, \textit{Internat. J. Algebra Comput.} \textbf{9}, (1999), 687-–
692.

\bibitem[Ch12]{ChaPhD}
V. V. Chaynikov, Properties of hyperbolic groups: free normal subgroups, quasiconvex subgroups and actions of maximal growth, Ph.D. Thesis, Vanderbilt University, (2012), available at http://etd.library.vanderbilt.edu/available/etd-06212012-172048/unrestricted/CHAYNIKOV.pdf

\bibitem[Co08]{Co09}
Y. de Cornulier, Infinite conjugacy classes in groups acting on trees,
\textit{Groups Geom. Dyn.} \textbf{3}, (2009), 267--277.

\bibitem[Di89]{Di90}
J. D. Dixon, Most finitely generated permutation groups are free,
\textit{Bull. London Math. Soc.} \textbf{22}, (1990), 222--226.

\bibitem[Fi12]{Fi12}
P. Fima, Amenable, transitive and faithful actions of groups acting on trees,
\textit{arXiv:1202.6467}, to appear in \textit{Annales de l'Institut Fourier}.

\bibitem[GG10]{GG11}
S. Garion and Y. Glasner, Highly transitive actions of $\rm{Out}(\mathbf{F}_n)$,
\textit{arXiv:1008.0563}, to appear in \textit{Groups, Geometry and Dynamics}.

\bibitem[GM90]{GM91}
A. M. W. Glass and S. H. McCleary, Highly transitive representations of free groups and free products,
\textit{Bull. Austral. Math. Soc.} \textbf{43}, (1991), 19--36.

\bibitem[Gu91]{Gu92}
S. V. Gunhouse, Highly transitive representations of free products on the natural numbers,
\textit{Arch. Math.} \textbf{58}, (1992), 435--443.

\bibitem[HW11]{HW11}
P. de la Harpe, C. Weber, Appendix by D. Osin, Malnormal subgroups and Frobenius groups: basics and examples,
\textit{arXiv:1104.3065}.

\bibitem[Hi90]{Hi92}
K. K. Hickin, Highly transitive Jordan representations of free products,
\textit{J. London Math. Soc.} \textbf{46}, (1992), 81--91.

\bibitem[Ki09]{Ki12}
D. Kitroser, Highly transitive actions of surface groups,
\textit{Proc. Amer. Math. Soc.} \textbf{140}, (2012), 3365--3375.

\bibitem[McD76]{McD77}
T. P. McDonough, A permutation representation of a free group,
\textit{Quart. J. Math. Oxford.} \textbf{28}, (1977), 353--356.

\bibitem[MS12]{MoSt}
S. Moon and Y. Stalder, Highly transitive actions of free products,
\textit{Algebr. Geom. Topol.} \textbf{13}, (2013), 589--607.

 \bibitem[Ne53]{NeumannCovered} B. H. Neumann, Groups covered by permutable subsets,
\textit{J. London Math. Soc.} \textbf{29}, (1954), 236--248.

\bibitem[Se77]{Se83}
J.-P. Serre, Arbres, amalgames, ${\rm SL}_2$,
\textit{Ast\'erisque} \textbf{46}, Soci\'et\'e Mathématique de France, (1977).

\end{thebibliography}
\end{document}